\documentclass[letterpaper, 10 pt, conference]{ieeeconf}
\IEEEoverridecommandlockouts
\usepackage{cite}
\usepackage{amsmath,amssymb,amsfonts}

\usepackage{algpseudocode}
\usepackage{algorithm}

\usepackage{graphicx}
\usepackage{textcomp}
\usepackage{breqn}
\usepackage{comment}
\usepackage{mathtools}
\usepackage{graphics}
\usepackage{bbold}

\usepackage{xcolor}

\newcommand{\tron}{\color{black}}
\newcommand{\troff}{\color{black}}

\newtheorem{thm}{Theorem}[section]

\newtheorem{assum}[thm]{Assumption}

\usepackage{hyperref}

\begin{document}

\title{State Aggregation for Distributed Value Iteration in Dynamic Programming}

\author{Nikolaus Vertovec
\and
Kostas Margellos\thanks{The paper is accepted in IEEE Control Systems Letters as \emph{N. Vertovec and K. Margellos, "State Aggregation for Distributed Value Iteration in Dynamic Programming,"}, the copyright of the published version is transferred to IEEE, and the published manuscript can be found in \url{https://ieeexplore.ieee.org/document/10149480}. The authors are with the Department of Engineering Science, University of Oxford. Email:\{nikolaus.vertovec, kostas.margellos\}@eng.ox.ac.uk}}

\maketitle

\begin{abstract}
    We propose a distributed algorithm to solve a dynamic programming problem with multiple agents, where each agent has only partial knowledge of the state transition probabilities and costs. We provide consensus proofs for the presented algorithm and derive error bounds of the obtained value function with respect to what is considered as the "true solution" obtained from conventional value iteration. To minimize communication overhead between agents, state costs are aggregated and shared between agents only when the updated costs are expected to influence the solution of other agents significantly. We demonstrate the efficacy of the proposed distributed aggregation method to a large-scale urban traffic routing problem. Individual agents compute the fastest route to a common access point and share local congestion information with other agents allowing for fully distributed routing with minimal communication between agents.
\end{abstract}

\IEEEpeerreviewmaketitle
\section{Introduction}
Value iteration is a well-established method for solving dynamic programming problems, yet it exhibits scalability issues for applications with a large state space. To this end, state aggregation can be used to reduce the number of states that need to be considered. The aggregation approach has a long history in scientific computing with applications ranging from the improvement of Galerkin methods \cite{Chatelin1982}, to solving large-scale optimization models \cite{Rogers1991}, and dynamic programming \cite{Bean1987}. In \cite{Tsitsiklis1996} it is shown how the aggregation approach can be used in conjunction with value iteration, however, for problems with large decision spaces and potentially cost and transition probabilities that evolve over time, conventional aggregation will still be hampered by scalability issues. To this end, we propose a multi-agent value iteration algorithm that utilizes aggregation to minimize communication overhead and allows for solving dynamic programming problems where each agent has only partial knowledge of the state transition costs and probability. 

\tron
Value iteration has been expanded to a multi-agent framework for the class of problems that contain a joint decision space \cite{BERTSEKAS2020}, yet in this approach no restrictions are imposed as far as the agents' knowledge of the underlying transition probabilities and costs are concerned. Such restrictions have been considered for small-scale problems as in \cite{Paul2022MultiAgentNR}, however, require a centralized critic for estimating the expected team benefit in a non-cooperative setting.

The proposed framework differs from many multi-agent reinforcement learning approaches that utilize sharing a weighted average of agents' local estimated value. Rather than considering a set of global states shared by all the agents in conjunction with a local reward observed by individual agents as in \cite{Guo2022, Doan2019}, we consider a Markov Decision Process (MDP) partitioned among agents with the objective of determining in a distributed manner the global value function without disclosing information on the transition probabilities and costs among agents.
\troff

This paper performs the following main contributions: (i) We propose a distributed value iteration algorithm for multi-agent dynamic programming, preventing transition probabilities and state costs from constituting global information among agents; (ii) We provide a rigorous mathematical analysis on the convergence and optimality properties of the proposed algorithm, merging tools from multi-agent consensus and dynamic programming principles; and, (iii) We demonstrate our scheme on a non-trivial traffic routing problem. 

The rest of the paper is organized as follows. Section \ref{sec:Problem Setup} provides some mathematical preliminaries and introduces the distributed multi-agent value iteration algorithm. Section \ref{sec:proofs} provides the main statements and associated mathematical analysis supporting the proposed algorithm. In Section \ref{sec:Application} we introduce a traffic routing problem and show how our proposed approach is able to solve the problem in a distributed manner. Finally, Section \ref{sec:Conclusion} provides concluding remarks and directions for future work.

\section{Problem Setup} \label{sec:Problem Setup}
\subsection{Multi-agent Markov Decision Processes}
We first introduce a standard MDP that will be expanded to a multi-agent setting in view of the proposed distributed algorithm presented in the sequel. We consider a finite set of $n$ states denoted by $\mathcal{X}$ and let $g(i,u,j)$ be the cost-to-go from state $i \in \mathcal{X}$ to state $j \in \mathcal{X}$, using the input $u \in \mathcal{U}(i)$, where $\mathcal{U}(i)$ is a finite set of actions/decisions available at state $i$. Furthermore, let $p(i,u,j)$ be the transition probability to transition from state $i$ to state $j$ using the input $u$, with $\sum_{j = 1}^n p(i,u,j) = 1$. The objective is to solve a global discounted infinite horizon dynamic programming problem. The associated Bellman equation is given by $J^*(i) = \min_{u \in U(i)} \sum_{j=1}^n p(i,u,j)[g(i,u,j) + \alpha J^*(j)]$, where $\alpha \in [0,1)$ is a discount factor, and $J^*$ denotes the optimal value function that satisfies the Bellman identity. Solving the Bellman equation using conventional value iteration or policy iteration requires knowledge of all transition probabilities as well as associated costs. 

We consider the setting where a set of $q$ agents collaborate to solve the aforementioned infinite horizon dynamic programming problem in a distributed manner, with only partial knowledge of the state transition probabilities and costs while minimizing communication between agents. As such, we partition the state-space $\mathcal{X}$ into $q$ subsets, $I_{\ell} \subset \mathcal{X}$, such that for all $m, \ell \in \{1, \ldots, q \}$ with $m \neq \ell$, $I_{\ell} \cap I_m = \emptyset$. Each agent knows only the state transition probabilities and cost for transitions originating within its state subset $I_{\ell}$, i.e., for each agent $\ell$, $g_\ell : I_{\ell} \times \mathcal{U} \times \mathcal{X} \rightarrow \mathbb{R}$, $p_\ell : I_{\ell} \times \mathcal{U} \times \mathcal{X} \rightarrow [0,1]$. Using conventional value iteration would require each agent to share its knowledge of the transition probabilities and costs, resulting in significant communication overhead. In the subsequent section, we instead introduce an alternative method relying on state aggregation that allows solving the discounted infinite horizon dynamic programming problem under consideration in a distributed manner where some tentative information is exchanged only with neighboring agents.

\subsection{Distributed Value Iteration}
We start by aggregating the value function (expected optimal cost-to-go) to construct for each agent $\ell=1,\ldots,q$, the aggregate value, defined as $r_{\ell,\ell} = \sum_{i \in I_{\ell}} d_{\ell i} J(i)$, where $J$ constitutes an approximation of the optimal value function of the Bellman equation, and $d_{\ell i}$ is the so-called disaggregation probability (encoding the contribution of each agent's value function to the aggregate one), satisfying for all $\ell=1,\ldots,q$, $\sum_{i \in I_{\ell}} d_{\ell i} = {} 1$ and for all $i \notin I_{\ell}$, $ d_{\ell i} = 0$.
The aggregate values for each agent, $r_{\ell,\ell}$, can be combined into a common vector denoted by $r_{\ell}$; tentative values for this vector will be communicated across all agents. Thus we define $r_{\ell} = \begin{bmatrix} r_{\ell,1}, \ldots, r_{\ell,\ell}, \ldots, r_{\ell,q} \end{bmatrix}^T$.
\tron
In the sequel, we will define an iterative scheme, according to which each agent $\ell = 1, \ldots q$ will update their estimate for the vector $r_{\ell}$; we denote this at iteration $k$, by $r_{\ell}^{k}$,  where its $m$-th element is indicated by $r_{\ell,m}^{k}$, $m = 1, \ldots, q$.\troff

Next, we introduce the aggregation probability satisfying for all $\ell=1,\ldots,q$,
\begin{equation*}
    \phi_{j\ell} = \begin{cases}
    1, & \text{if } j \in I_{\ell}, \\
    0, & \text{otherwise}.
    \end{cases}
\end{equation*}
Such a formulation of the  aggregation probability is known as hard aggregation in the literature \cite[p. 311]{bertsekas2019}.

To solve the discounted infinite horizon dynamic programming problem, we propose a distributed algorithm for which the pseudocode is given in Algorithm \ref{alg:AgentComms} and \ref{alg:AgentUpdate}. \tron The proposed algorithm involves an agent-to-agent communication protocol. At each algorithm iteration $k \geq 0$ we consider the directed communication graph $(V,E(k))$, where the node set $V = \{1,\ldots,q\}$ includes the agents and the set $E(k)$, the directed edges $(m,l)$, indicating that at iteration $k$ agent $\ell$ can receive information from agent $m$. Let $E_B(k) = \bigcup_{i=k}^{(k+B)} E(i)$ for some integer $B$. We make the fairly standard assumption in the existing literature on the communication structure between agents \cite{Nedic2009, Doan2019}:
\begin{assum} \label{ass:Network}
    [Connectivity and Communication] There exists a positive integer $B$ such that $(V,E_B(k))$ is fully connected for each iteration $k$.
\end{assum}
Assumption \ref{ass:Network} implies that for any agent pair $(\ell,m)$ there is a direct link at least once every $B$ iterations. This prevents agents from having to share information with all other agents at all iterations as well as for the presence of a central authority. 
\troff

Initially in Algorithm \ref{alg:AgentComms}, each agent $\ell$, $\ell = 1, \ldots, q$, starts with some tentative value of the aggregate vector $r_{\ell}^{0}$, and local cost function $V_{\ell}^{0}$. The initialization choice is arbitrary.
Utilizing the aggregate vector $r_{\ell}^{k}$ at iteration $k \geq 0$, each agent $\ell$ locally and in parallel executes Algorithm \ref{alg:AgentUpdate} so as to construct the updated local value function $V_{\ell}^{k+1}: I_{\ell} \rightarrow \mathbb{R}$ which will, in turn, yield an updated aggregated cost $r_{\ell,\ell}^{k+1} \in \mathbb{R}$ (Algorithm \ref{alg:AgentComms}, Steps \ref{alg2:line:func}).

\tron
To minimize communication overhead, an updated aggregated cost $r_{\ell,\ell}^{k+1}$ is only shared with other agents in the set $N_\ell(k)$ (the neighbors of $\ell$ at iteration $k$, when the aggregated cost has changed significantly, i.e., $\|r_{\ell,\ell}^{k+1} - r_{\ell,\ell}^{\mathrm{prev}}\| > C_{\mathrm{threshold}}$, where $r_{\ell,\ell}^{\mathrm{prev}}$ is the aggregated cost previously broadcast to other agents and $C_{\mathrm{threshold}}$ is a communication threshold (Algorithm \ref{alg:AgentComms}, Steps \ref{alg2:line:broadcast_start} - \ref{alg2:line:broadcast_end}), or when agent $\ell$ and a neighboring agent $m$ have not communicated in the last $B$ iterations (see Step 11; this ensures that there exists a bounded intercommunication time as required by Assumption \ref{ass:Network}). For further discussion on the communication threshold to limit the transmission of insignificant data, we refer to \cite{Gatsis22}.
\troff

Until convergence is reached, each agent will repeat the execution of Algorithm \ref{alg:AgentUpdate} and the subsequent communication step. Note that at iteration $k$, the vector $r_{\ell}^{k}$ may contain aggregate values from other agents received at prior time steps (Algorithm \ref{alg:AgentComms}, Step \ref{alg2:line:oldValues}). It will be shown in the proof of Theorem \ref{thm:consensus} that convergence will nevertheless be reached.

We now turn to the update of $V_{\ell}^{k}$ and $r_{\ell,\ell}^{k}$, performed when calling Algorithm \ref{alg:AgentUpdate}. Each agent uses the local value function, $V_{\ell}^{k}$, to represent the cost-to-go function at each state $i \in I_{\ell}$, and the aggregated cost, $r_{\ell}^{k}$, to approximate the cost for states $j \in \mathcal{X} \backslash I_{\ell}$. The local value function is iteratively updated for each state $i \in I_{\ell}$ (Algorithm \ref{alg:AgentUpdate}, Steps \ref{alg1:line:V_update_start} - \ref{alg1:line:V_update_end}). This update follows from standard value iteration; choosing the control action at state $i$ that minimizes the expected cost-to-go (Algorithm \ref{alg:AgentUpdate}, Step \ref{alg2:line:bellmanoperator}) and using a Gauss-Seidel update on the local value function to improve convergence and minimize memory usage (Algorithm \ref{alg:AgentUpdate}, Step \ref{alg2:line:gauss-seidel}). We perform an update only to the local value function since we restrict ourselves to states for which sufficient knowledge of the transition probabilities and costs is available. After the local value function has been updated, the updated local aggregated cost, $r_{\ell,\ell}^{k+1}$, is computed (Algorithm \ref{alg:AgentUpdate}, Step \ref{alg1:line:r_l}).

Notice that Algorithms \ref{alg:AgentComms} and \ref{alg:AgentUpdate} prevent disclosing the transition probabilities to all agents; in contrast, each agent $\ell$, has access only to the probabilities and costs associated with transitions originating from their partition $I_{\ell}$, $\ell=1,\ldots,q$.

\tron
\begin{algorithm}
\caption{Distributed value iteration}
\begin{algorithmic}[1]
    \tron
    \State \textbf{Initialization}
    \State Set $r_{\ell}^{0} \leftarrow [0, \ldots, 0]$ for all $\ell = {1, \ldots, q}$
    \State Set $r_{\ell,\ell}^{\mathrm{prev}} \leftarrow 0$ for all $\ell = {1, \ldots, q}$
    \State Set $V_{\ell}^{0}(i) \leftarrow 0$ for all $i \in I_{\ell}$ and $\ell = {1, \ldots, q}$
    \State $k = 0$
\Repeat \label{alg2:line:repeat}
\For{each Agent $\ell \in {1, \ldots, q}$} 
    \State $V_{\ell}^{k+1}, r_{\ell,\ell}^{k+1} \leftarrow $Agent-Update($V_{\ell}^{k}, r_{\ell}^{k}$) \label{alg2:line:func}
    \State $N_\ell(k) = \{m \in \{1,\ldots,q\} :~ (\ell,m) \in E(k) \}$
    \If{$\|r_{\ell,\ell}^{k+1} - r_{\ell,\ell}^{\mathrm{prev}}\| > C_{\mathrm{threshold}}$, \\ \hspace{1.6cm}\textbf{or if} $(\ell,m) \notin \bigcup_{i=k-B+1}^k E(i)$} \label{alg2:line:broadcast_start}
        \State Send $r_{\ell,\ell}^{k+1}$ to all agents $m \in N_\ell(k)$ 
    \EndIf \label{alg2:line:broadcast_end}
    \If{receiving information from an agent $m$}
        \State $r_{\ell,m}^{k+1} \leftarrow r_{m,m}^{k+1}$
    \Else
        \State $r_{\ell,m}^{k+1} \leftarrow r_{\ell,m}^{k}$ \label{alg2:line:oldValues}
    \EndIf
\EndFor 
\Until{$\|r_{\ell}^{k+1} - r_{\ell}^{k}\|\leq \mathrm{tolerence}$ for all $\ell = {1, \ldots, q}$}
\end{algorithmic}
\label{alg:AgentComms}
\end{algorithm}
\troff

\begin{algorithm}
\caption{Agent-Update}
\begin{algorithmic}[1]
\Function{Agent-Update}{$V_{\ell}^{k}, r_{\ell}^{k}$}
    \For{$i \in I_{\ell}$} \label{alg1:line:V_update_start}
        \State \hspace{-0.3cm} $V_{\ell}^{k+1}(i) \leftarrow \min_{u \in U(i)} \sum_{j=1}^n p_\ell(i,u,j)[g_\ell(i,u,j)$ \\
        \hspace{3cm} $+ \alpha (\phi_{j\ell} V_{\ell}^{k}(j) + \sum_{\substack{m=1 \\ m\neq \ell}}^q \phi_{jm} r_{\ell,m}^{k})]$ \label{alg2:line:bellmanoperator}
        \State \hspace{-0.3cm} $V_{\ell}^{k}(i) \leftarrow V_{\ell}^{k+1}(i)$ \label{alg2:line:gauss-seidel}
    \EndFor \label{alg1:line:V_update_end}
    \State $r_{\ell,\ell}^{k+1} \leftarrow \sum_{j \in I_{\ell}} d_{\ell j} V_{\ell}^{k+1}(j)$ \label{alg1:line:r_l}
    \State \textbf{return} $V_{\ell}^{k+1}, r_{\ell,\ell}^{k+1}$
\EndFunction
\end{algorithmic}
\label{alg:AgentUpdate}
\end{algorithm}

\section{Algorithm analysis} \label{sec:proofs}
The following theorem is the main result of this section.
\begin{thm} Consider Assumption \ref{ass:Network}.
Algorithm \ref{alg:AgentComms} converges asymptotically to a common value for $r$ among agents, i.e., there exists $\bar{r}$ such that for all $\ell=1,\ldots,q$, $\lim_{k \to \infty} \|r_{\ell}^k - \bar{r}\| = 0$. 
Moreover, for all $\ell=1,\ldots,q$, for all $i \in \mathcal{X}$, $V_{\ell}^k(i)$ converges to some $\overline{V}_\ell^*(i)$.
\label{thm:consensus}
\end{thm}


\begin{proof}
We show convergence when agents transmit $r_{\ell,\ell}^k$, $\ell=1,\ldots,q$, irrespective of whether this deviates more than $C_{\mathrm{threshold}}$ from their previously transmitted cost.
This is without loss of generality as in this case the resulting sequence of transmitted costs would form a subsequence of $\{r_{\ell,\ell}^{k}\}_{k=0}^\infty$, hence it will be convergent as we will show that $\{r_{\ell,\ell}^{k}\}_{k=0}^\infty$ converges.

\tron
Fix any $k \geq B$, and any $\ell \in \{1,\ldots,q\}$. Let $\epsilon_{\ell,m}^{k} = r_{\ell,m}^{k+1} - r_{\ell,m}^{k}$,
$\epsilon_{\ell,\infty}^{k} = \max_{m = 1, \ldots, q} |r_{\ell,m}^{k+1} - r_{\ell,m}^{k}|,$ and
    $\epsilon_{\infty}^{k} = \max_{\ell = 1,\ldots,q} \sum_{n = k-B}^k \epsilon_{\ell,\infty}^{n}$,
i.e., $\epsilon_{\infty}^{k}$ is the maximum among agents of the cumulative incremental update over the most recent $B+1$ iterations; recall that due to Assumption \ref{ass:Network} this is the window within which agent $\ell$ communicates with other agents at least once. \troff
  
  Consider the update to the value $V_{\ell}^{k}(i)$, for an arbitrary iteration $k$ and agent $\ell=1,\ldots,q$, which we define as 
\begin{align*}
    &\delta_\ell^{k}(i) = V_{\ell}^{k+1}(i) - V_{\ell}^{k}(i) 
    = \min_{u \in U(i)} \sum_{j=1}^n p_\ell(i,u,j)\Big(g_\ell(i,u,j) \\ & \qquad  + \alpha [\phi_{j\ell} V_{\ell}^{k}(j) + \sum_{\substack{m=1 \\ m\neq \ell}}^q \phi_{jm} r_{\ell,m}^{k}]\Big) - V_{\ell}^{k}(i)\\
    %
    %
    & = \min_{u \in U(i)}\Big( \sum_{j=1}^n p_\ell(i,u,j)\Big(g_\ell(i,u,j) + \alpha [\phi_{j\ell} V_{\ell}^{k-1}(j) \\ & \qquad + \sum_{\substack{m=1 \\ m\neq \ell}}^q \phi_{jm} r_{\ell,m}^{k-1}]\Big) + \alpha\sum_{j=1}^n p_\ell(i,u,j)\Big(\phi_{j\ell}\delta_\ell^{k-1}(j) \\ & \qquad  +\sum_{\substack{m=1 \\ m\neq \ell}}^q \phi_{jm} \epsilon_{\ell,m}^{k-1} \Big)\Big) - V_{\ell}^{k}(i) \\
    & \leq V_{\ell}^{k}(i) + \max_{u \in U(i)} \alpha\sum_{j=1}^n p_\ell(i,u,j)\Big(\phi_{j\ell}\delta_\ell^{k-1}(j) \\ & \quad +\sum_{\substack{m=1 \\ m\neq \ell}}^q \phi_{jm} \epsilon_{\ell,m}^{k-1} \Big) - V_{\ell}^{k}(i)\\
    & = \max_{u \in U(i)} \alpha\sum_{j=1}^n p_\ell(i,u,j)\Big(\phi_{j\ell}\delta_\ell^{k-1}(j) +\sum_{\substack{m=1 \\ m\neq \ell}}^q \phi_{jm} \epsilon_{\ell,m}^{k-1} \Big),
\end{align*}
where the first equality follows from the definition of $V_{\ell}^k$, and the second one from a rearrangement after substituting $V_{\ell}^{k}(j) = V_{\ell}^{k-1}(j) + \delta_\ell^{k-1}(j)$ and $r_{\ell,m}^{k} = r_{\ell,m}^{k-1} + \epsilon_{\ell,m}^{k-1}$. The inequality follows from the definition of $V_{\ell}^k$ and by considering the maximum over $u\in U(i)$.

For each $\ell=1, \ldots, q$, we now define the maximum update of $|\delta_\ell^{k}|$, over all states $i \in I_{\ell}$, as
$\overline{\delta}_\ell^{k} = \max_{i \in I_{\ell}} |\delta_\ell^{k}(i)|$.
It then follows that for all $i\in I_{\ell}$,
\begin{align}
    & \delta_\ell^{k}(i) 
    %
    \leq \max_{u \in U(i)} \alpha\sum_{j=1}^n p_\ell(i,u,j)\Big(\phi_{j\ell}\overline{\delta}_\ell^{k-1} +\sum_{\substack{m=1 \\ m\neq \ell}}^q \phi_{jm} \epsilon_{\ell,\infty}^{k-1} \Big) \nonumber \\
    & \leq \max_{u \in U(i)} \alpha \sum_{j=1}^n p_\ell(i,u,j)\Big(\sum_{m=1}^q \phi_{jm} \max\{\overline{\delta}_\ell^{k-1},\epsilon_{\ell,\infty}^{k-1}\} \Big) \nonumber \\
    & = \alpha\max\{\overline{\delta}_\ell^{k-1},\epsilon_{\ell,\infty}^{k-1}\},\label{eq:delta_bound}
\end{align}  
where the first inequality follows from the definition of $\overline{\delta}_\ell^{k}$, and the equality follows from the fact that $\sum_{j=1}^n p(i,u,j) = 1$, and $\sum_{m=1}^q \phi_{jm} = 1$.  \tron
Define $\overline{\delta}_{\infty}^{k} \coloneqq \max_{\ell = 1, \ldots, q} \sum_{n = k-B}^k \overline{\delta}_\ell^{n}$, resulting in
\begin{align}
    \overline{\delta}_{\infty}^{k} & \leq  \alpha \max_{\ell = 1, \ldots, q} \sum_{n = k-B}^k
 \max\{\overline{\delta}_\ell^{k-1},\epsilon_{\ell,\infty}^{k-1}\} \nonumber \\ 
 & = \alpha  
 \max\{\overline{\delta}_{\infty}^{k-1},\epsilon_{\infty}^{k-1}\},
 \label{eq:delta_inf}
\end{align}
where the inequality is since \eqref{eq:delta_bound} holds for all $i \in I_{\ell}$, and the exchange of the summation and the maximization operator is since all quantities are non-negative.
\tron
At the same time, 
\begin{align}
  &\sum_{n = k-B}^k \max_{m = 1, \ldots, q} |r_{\ell,m}^{n+1} - r_{\ell,m}^{n}|\nonumber \\
    & \leq \sum_{n = k-B}^k \max_{m = 1, \ldots, q}  |r_{m,m}^{n+1} - r_{m,m}^{n}|\nonumber \\
    & =  \sum_{n = k-B}^k \max_{m = 1, \ldots, q}|\sum_{i \in I_m} d_{mi} (V_m^{n}(i) + \delta_m^{n}(i)) - r_{m,m}^{n}| \nonumber\\
    %
    & = \sum_{n = k-B}^k \max_{m = 1, \ldots, q}  |\sum_{i \in I_m} d_{mi}\delta_m^{n}(i)| \nonumber \\
    &= \sum_{n = k-B}^k \max_{m = 1, \ldots, q} \overline{\delta}_m^n
    = \max_{m = 1, \ldots, q} \sum_{n = k-B}^k  \overline{\delta}_m^n = \overline{\delta}_{\infty}^k, 
    \label{eq:eps_ell_inf}
\end{align}
where the inequality is since $\sum_{n = k-B}^k |r_{\ell,m}^{n+1} - r_{\ell,m}^{n}|$ is bounded by the value this would become if agents $\ell$ and $m$ communicated at every iteration $n = k-B+1,\ldots,k$ (see Steps 14-18, Algorithm \ref{alg:AgentComms}). The
first equality follows from the fact that $r_{m,m}^{n+1} = \sum_{i \in I_m} d_{mi} V_m^{n+1}(i)$ (see Step \ref{alg1:line:r_l}, Algorithm \ref{alg:AgentUpdate}), and since $V_m^{n+1}(i) = V_m^{n}(i) + \delta_m^{n}(i)$, while the last one is due to the definition of $\overline{\delta}_{\infty}^k$.

By \eqref{eq:eps_ell_inf} we can upper-bound $\epsilon_{\infty}^k = \max_{\ell=1,\ldots,q}\sum_{n = k-B}^k \max_{m = 1, \ldots, q} |r_{\ell,m}^{n+1} - r_{\ell,m}^{n}|$ as 
\begin{align}
    \epsilon_{\infty}^k & \leq \max_{\ell=1,\ldots,q} \overline{\delta}_{\infty}^k \leq \alpha  
 \max\{\overline{\delta}_{\infty}^{k-1},\epsilon_{\infty}^{k-1}\},
    \label{eq:eps_inf}
\end{align}
where the last inequality follows from \eqref{eq:delta_inf}, and the fact that the bound in \eqref{eq:delta_inf} is independent of $\ell$. \troff

By \eqref{eq:delta_inf} and \eqref{eq:eps_inf} we then have that $\max\{\overline{\delta}_\infty^{k}, \epsilon_\infty^{k}\} \leq \alpha \max\{\overline{\delta}_\infty^{k-1}, \epsilon_\infty^{k-1}\}$, which implies that $\max\{\overline{\delta}_\infty^{k}, \epsilon_\infty^{k}\}$ is contractive. As a result, $\{\max\{\overline{\delta}_\infty^{k}, \epsilon_\infty^{k}\}\}_{k\geq 0}$, and hence also $\{\epsilon_\infty^{k}\}_{k\geq0}$ and $\{\overline{\delta}_{\infty}^k\}_{k\geq 0}$ will be converging to zero. Since $\{\epsilon_\infty^{k}\}_{k\geq0}$ converges, then each term in the summation in the definition of $\epsilon_\infty^{k}$ will also converge, implying that $\{|r_{\ell,m}^{k+1} -r_{\ell,m}^{k}|\}_{k\geq 0}$ is convergent. This implies that for all $\ell=1,\ldots,q$ there exists $\bar{r}$ such that $\lim_{k \to \infty} \|r_{\ell}^k - \bar{r}\| = 0$, while convergence of $\{\overline{\delta}_{\infty}^k\}_{k\geq 0}$ (due to \eqref{eq:delta_inf}, and the definition of $\delta_\ell^k(i)$) implies that for all $\ell=1,\ldots,q$, for all $i\in \mathcal{X}$, $\{V_{\ell}^k(i)\}_{k\geq 0}$ would be convergent, thus concluding the proof.
\end{proof}
Theorem \ref{thm:consensus} implies consensus among agents to a common $\bar{r}$, and also establishes convergence of $\{V_{\ell}^k(i)\}_{k \geq 0}$ to some $V_{\ell}^*(i)$. \tron Moreover, the convergence rate is linear since $\max\{\overline{\delta}_\infty^{k}, \epsilon_\infty^{k}\}$ is shown to be contractive in the proof of Theorem \ref{thm:consensus}. The exact convergence rate for $\{|r_{\ell,m}^{k+1} -r_{\ell,m}^{k}|\}_{k\geq 0}$ will depend on $B$. \troff

Next, we consider error bounds between the limiting $V_{\ell}^*(i)$ and the optimal $J^*(i)$, satisfying the Bellman equation. For the subsequent result we assume that $B=0$, i.e., agents communicate with all other agents at all iterations.

\begin{thm}
Consider Assumption \ref{ass:Network} with $B=0$ and Algorithm \ref{alg:AgentComms}. For all $\ell=1,\ldots,q$, $i \in \mathcal{X}$, we have that the limit point $V_{\ell}^*(i)$ of $\{V_{\ell}^k(i)\}_{k \geq 0}$ satisfies
\begin{equation}
    |V_{\ell}^*(i) - J^*(i)| \leq \alpha \frac{\delta}{1-\alpha}, \label{eq:errorBounds}
\end{equation}
where $\delta  \coloneqq \max_{\ell \in \{1, \ldots, q\}}\max_{i,j \in I_{\ell}} |J^*(i) - J^*(j)|$ and $J^*$ is the solution of the Bellman equation.
\label{thm:errorBounds}
\end{thm}
\begin{proof}
The proof is inspired by \cite{Tsitsiklis1996}. Fix $\ell=1,\ldots,q$ and $i\in \mathcal{X}$. We only show that $V_{\ell}^*(i) \leq J^*(i) + \alpha \frac{\delta}{1-\alpha}$, as the other side in \eqref{eq:errorBounds} follows symmetric arguments.
To this end, for each $i \in I_{\ell}$, define
$\overline{V}_\ell(i) \coloneqq J^*(i) + \alpha \frac{\delta}{1-\alpha}$.
It follows then that the aggregate of $\overline{V}_\ell(i)$ can be constructed as
\begin{align}
    \overline{r}_{\ell,\ell} & \coloneqq \sum_{j \in I_{\ell}} d_{\ell j}\overline{V}_\ell(j) = \sum_{j \in I_{\ell}} d_{\ell j} (J^*(j) + \alpha \frac{\delta}{1-\alpha} ) \nonumber \\
    & \leq
    \sum_{j \in I_{\ell}} d_{\ell j}(\min_{i \in I_{\ell}} J^*(i) + \delta + \alpha \frac{\delta}{1-\alpha} ) \nonumber \\
    & = \min_{i \in I_{\ell}} J^*(i) + \delta + \alpha \frac{\delta}{1-\alpha} = \min_{i \in I_{\ell}} J^*(i) + \frac{\delta}{1-\alpha}, \label{eq:r_bar}
\end{align}
where the inequality is since $J^*(j) \leq \min_{i \in I_{\ell}} J^*(i) + \delta$, while the second last equality is due to $\sum_{j \in I_{\ell}} d_{\ell j} = 1$.

Denote the Bellman operator induced by the Bellman equation as $T$, such that we can compactly write the Bellman equation as $J^*(i) = (TJ^*)(i)$, where effectively by $(TJ^*)(i)$ we imply the right-hand side of the Bellman equation which depends on $i$ and on $J^*(j)$ for all $j=1,\ldots,n$.
For all $i=1,\ldots,n$, we now consider the application of the Bellman operator to $\overline{V}_\ell$, i.e.,
\begin{align}
    & (T \overline{V}_\ell)(i) = \min_{u \in U(i)} \sum_{j=1}^n p_\ell(i,u,j)\Big(g_\ell(i,u,j) +\alpha [\phi_{j\ell} \overline{V}_\ell(j) \nonumber \\ & \qquad 
    + \sum_{\substack{m=1 \\ m\neq \ell}}^q \phi_{jm} \overline{r}_{\ell,m}]\Big) \nonumber \\
    & \leq \min_{u \in U(i)} \sum_{j=1}^n p_\ell(i,u,j)\Big(g_\ell(i,u,j) + \alpha [\phi_{j\ell}  (J^*(j) \nonumber \\ & \qquad 
    + \alpha\frac{\delta}{1-\alpha}) + \sum_{\substack{m=1 \\ m\neq \ell}}^q \phi_{jm} (\min_{j \in I_m} J^*(j) + \frac{\delta}{1-\alpha})]\Big) \nonumber \\
    %
    %
    & \leq \min_{u \in U(i)} \sum_{j=1}^n p_\ell(i,u,j)\Big(g_\ell(i,u,j) \nonumber \\
    &\qquad+ \alpha \sum_{m=1}^q \phi_{jm}(J^*(j) + \frac{\delta}{1-\alpha})\Big) \nonumber \\
    & \leq \min_{u \in U(i)} \sum_{j=1}^n p_\ell(i,u,j)\Big(g_\ell(i,u,j) + \alpha J^*(j)\Big) \nonumber \\
    & \qquad + \alpha \frac{\delta}{1-\alpha}
    = J^*(i) + \alpha \frac{\delta}{1-\alpha}
      = \overline{V}_\ell(i), \label{eq:V_mon}
\end{align}
where the first inequality is due to the definition of $\overline{V}_\ell$ and utilizes $\overline{r}_{\ell,m} = \overline{r}_{m,m}$, which holds since $B=0$. The second inequality follows, since $\min_{j \in I_m} J^*(j) \leq J^*(j)$ for all $j \in \mathcal{X}$, and $\alpha\frac{\delta}{1-\alpha} \leq \frac{\delta}{1-\alpha}$, that in turn allows us to combine the terms multiplied with $\phi_{j\ell}$ into the last summation. The last inequality follows from $\sum_{m=1}^q \phi_{jm} =1$, while the second last equality is due to the Bellman equation.

By \eqref{eq:V_mon} we have that $(T \overline{V}_\ell)(i) \leq \overline{V}_\ell(i)$, for all $i=1,\ldots,n$, which implies that $\{\overline{V}_\ell^k(i)\}_{k \geq 0}$ is a non-increasing sequence. Moreover, $T$ is contractive and as such it will converge to its (unique) fixed point; a direct consequence of Theorem \ref{thm:consensus} is that this fixed point is $V_{\ell}^*(i)$ (as the latter was constructed by successive applications of the Bellman operator). Therefore, for all $i=1,\ldots,n$, $V_{\ell}^*(i) = \lim_{k \to \infty} (T^k \overline{V}_\ell)(i) \leq \overline{V}_\ell(i)$, thus establishing $V_{\ell}(i) \leq J^*(i) + \alpha \frac{\delta}{1-\alpha}$, concluding the proof.
\end{proof}

It follows from Theorem \ref{thm:errorBounds}, that if the aggregation sets $I_{1}, \ldots, I_{q}$ are chosen such that the cost function $J^*$ is expected to vary moderately between states within an aggregation set, then the maximum error compared to $V_{1}, \ldots, V_{q}$ will be moderate as well. One method to aggregate states with similar costs is to use feature-based aggregation, whereby the aggregation is performed on a set of representative features instead of representative states. This form of aggregation has been thoroughly investigated; we refer to \cite[pg. 322]{bertsekas2019}, \cite[pg. 68]{bertsekas1996} and references therein. For applications with a high discount factor, i.e., $\alpha << 1$, the maximum error will also be moderate.

\section{Application to traffic routing} \label{sec:Application}
\subsection{Simulation Set-up}
We demonstrate the proposed algorithm in a traffic routing case study. To this end, we begin by modeling the traffic network as a graph, with vertices representing junctions and edges representing roads connecting junctions. An illustration of such a graph representation for the Oxford road network is shown in Figure \ref{fig:oxford_graph}.

\begin{figure}
 	\centering
	\includegraphics[width=0.45\textwidth]{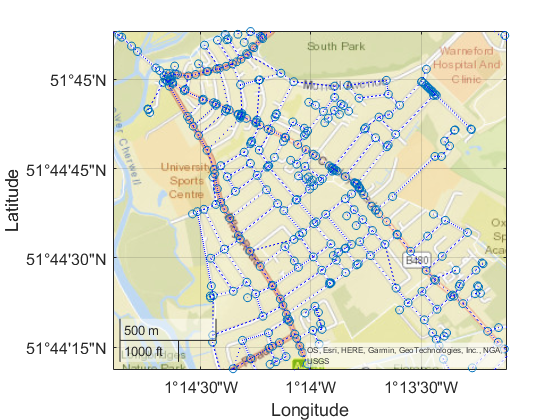}
	\caption{Graph representation of the East-Oxford road network. Junctions are represented by vertices and edges represent roads.}
	\label{fig:oxford_graph}
\end{figure}

\tron
We consider a low-energy antenna with a limited range and computation power to be situated at the center of each network partition, gathering the current speed and location of nearby vehicles. The transition cost of an edge is the expected travel time along that edge and is computed based on the data received by cars, i.e.,
\begin{equation}    
    g_\ell(i,u,j) = \frac{\text{length of edge}}{\text{average speed of cars on edge}}, 
    \label{eqn:cost-to-go}
\end{equation}
where $u$ is used to determine the selected edge between node $i$ and $j$. Due to the nature of the low-energy application, we utilize the proposed algorithm to limit the amount of data that needs to be sent over longer distances.



As is common in transit node routing, vehicles are trying to reach a common access node, such as a freeway used for long-distance routing. The goal is to find the fastest path to the access point, thus the cost at each vertex is the discounted expected travel time to the access point. In our example related to the Oxford traffic network, the access point will be London Road, leading long-distance travelers toward London. 
\troff

\subsection{Simulation results}
\tron The Oxford road network is partitioned into 5 subgraphs using K-means clustering of the vertices based on their euclidean distance, as shown in Figure \ref{fig:clustering}. \troff

\begin{figure}
 	\centering
	\includegraphics[width=0.45\textwidth]{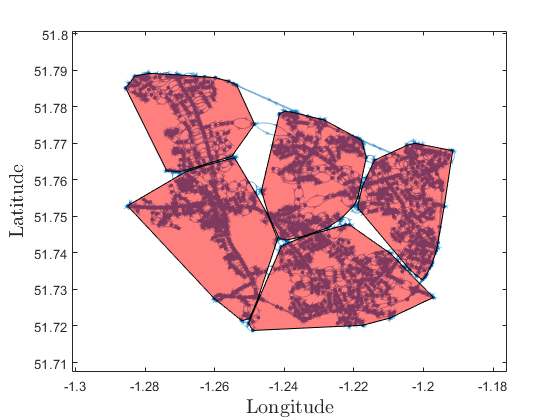}
	\caption{The complete Oxford road network is subdivided into 5 subgraphs representing the aggregation sets.}
	\label{fig:clustering}
\end{figure}

The cost-to-go is calculated as in \eqref{eqn:cost-to-go}, where the average car speed is randomly generated from a uniform distribution to lie between $25\%$ and $100\%$ of the speed limit of the relevant road. For each vertex, the outgoing edges are enumerated and the transition probability between two vertices is set to either 1 or 0, depending on whether an edge directly connects the vertices and the input selects that edge.

The disaggregation probabilities are zero for all nodes that have no edge connected to a vertex in another aggregation set $I_{\ell}$. The remaining vertices are given a normalized non-zero disaggregation probability. The discount factor,  $\alpha = 0.9$, is chosen close enough to 1 so as to reflect the desire to reach the final node and avoid loops, yet smaller than 1, so as to weight costs further off as less important due to the constantly evolving traffic situation.

\tron Solving the aggregate problem and comparing it to what is considered to be the "true solution" $J^*$ obtained via conventional value iteration, we notice that the normalized average error of the expected cost-to-go, i.e., $\frac{1}{n}\sum_{\ell = 1, \ldots, q} \sum_{i \in I_{\ell}}\frac{|\phi_{il} V_{\ell}(i) - J^*(i)|}{|J^*(i)|}$, is $0.94\%$, and the normalized maximum error i.e., $\max_{\ell = 1, \ldots, q} \max_{i \in I_{\ell}} \frac{|\phi_{il} V_{\ell}(i) - J^*(i)|}{|J^*(i)|}$, is $190.83\%$.

The value function is shown in Figure \ref{fig:cost}. The evolution over time of the aggregated costs is shown in Figure \ref{fig:communication}, where the colored dots represent when the change to an agent's local value function compared to the last broadcast is greater than the chosen communication threshold ($0.1$ in this example). \troff At this point, the agent will broadcast its updated aggregated cost. If since the last broadcast, no agents value function has changed significantly, the agents signal each other that convergence is reached and the algorithm terminates.
\tron
For comparison, we show how, on average, the normalized error increases as we increase the number of agents (see Table below). This is to be expected, as agents will rely more heavily on the aggregated values the smaller their respective partitions become.
    \begin{table}[h]
    \tron
    \begin{tabular}{l|llll}
    Number of Agents         & 4         & 8         & 12        & 16        \\ \hline
    Normalized average error & $0.67\%$ & $1.63\%$ & $2.84\%$ & $4.46\%$
    \end{tabular}
    \label{table}
    \end{table}
\troff
\begin{figure}
\tron 
 	\centering
	\includegraphics[width=0.45\textwidth]{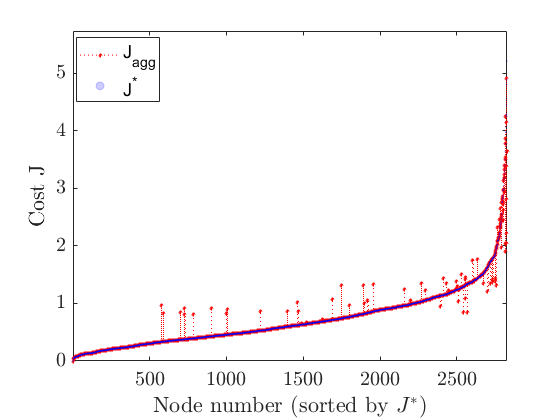}
	\caption{The node cost sorted by size. The blue line indicates the "true cost" computed using conventional value iteration, the red markers represent the cost obtained by the proposed distributed algorithm.}
	\label{fig:cost}
\troff
\end{figure}
\begin{figure}
\tron
 	\centering
	\includegraphics[width=0.45\textwidth]{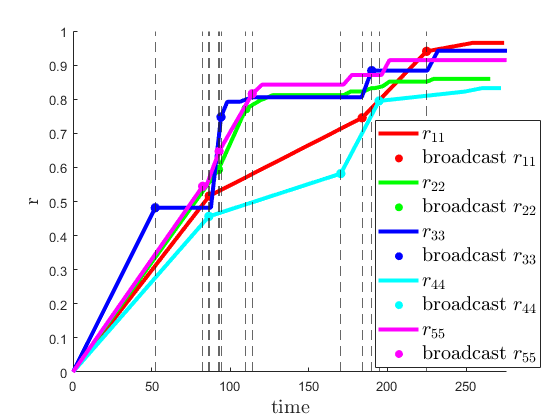}
	\caption{The aggregated cost $r$ evolving over time with minimal communication between agents.}
	\label{fig:communication}
\troff
\end{figure}

To reproduce the numerical results the associated code has been made available in \cite{VertovecCode2022}, with the ability to upload any OpenStreetMap file, which is then converted to a graph and subsequently a discounted Markov decision process.

\section{Conclusion}\label{sec:Conclusion}
We presented a multi-agent extension of aggregated value iteration which was shown to be able to solve large-scale dynamic programming problems in a fully distributed manner. The presented methodology finds application in problems where each agent has only partial knowledge of the transition probabilities and costs. To this end, we demonstrated its efficacy in a distributed traffic routing problem, for which the code has been made available in \cite{VertovecCode2022}. 

Future work aims at extending Theorem \ref{thm:errorBounds} that is based on the additional assumption of full network connectivity at all iterations to the more general case of Assumption \ref{ass:Network}, as well as to the case of hard aggregation (where each state is clearly assigned to no more than one agent). 

\section*{Acknowledgments}
The authors would like to acknowledge Dr. Konstantinos Gatsis for several insightful discussions.

\bibliographystyle{./bibtex/IEEEtran}
\bibliography{./bibtex/IEEEabrv,./bibtex/IEEEexample,./bibtex/bibliography}

\begin{thebibliography}{10}
\providecommand{\url}[1]{#1}
\csname url@samestyle\endcsname
\providecommand{\newblock}{\relax}
\providecommand{\bibinfo}[2]{#2}
\providecommand{\BIBentrySTDinterwordspacing}{\spaceskip=0pt\relax}
\providecommand{\BIBentryALTinterwordstretchfactor}{4}
\providecommand{\BIBentryALTinterwordspacing}{\spaceskip=\fontdimen2\font plus
\BIBentryALTinterwordstretchfactor\fontdimen3\font minus
  \fontdimen4\font\relax}
\providecommand{\BIBforeignlanguage}[2]{{%
\expandafter\ifx\csname l@#1\endcsname\relax
\typeout{** WARNING: IEEEtran.bst: No hyphenation pattern has been}%
\typeout{** loaded for the language `#1'. Using the pattern for}%
\typeout{** the default language instead.}%
\else
\language=\csname l@#1\endcsname
\fi
#2}}
\providecommand{\BIBdecl}{\relax}
\BIBdecl

\bibitem{Chatelin1982}
F.~Chatelin and W.~L. Miranker, ``Acceleration by aggregation of successive
  approximation methods,'' \emph{Linear Algebra and its Applications}, vol.~43,
  pp. 17--47, 1982.

\bibitem{Rogers1991}
D.~F. Rogers, R.~D. Plante, R.~T. Wong, and J.~R. Evans, ``Aggregation and
  disaggregation techniques and methodology in optimization,'' \emph{Operations
  Research}, vol.~39, no.~4, pp. 553--582, 1991.

\bibitem{Bean1987}
J.~C. Bean, J.~R. Birge, and R.~L. Smith, ``Aggregation in dynamic
  programming,'' \emph{Operations Research}, vol.~35, no.~2, pp. 215--220,
  1987.

\bibitem{Tsitsiklis1996}
J.~N. Tsitsiklis and B.~van Roy, ``Feature-based methods for large scale
  dynamic programming,'' \emph{Machine Learning}, vol.~22, no.~1, pp. 59--94,
  Mar 1996.

\bibitem{BERTSEKAS2020}
D.~Bertsekas, ``Multiagent value iteration algorithms in dynamic programming
  and reinforcement learning,'' \emph{Results in Control and Optimization},
  vol.~1, p. 100003, 2020.

\bibitem{Paul2022MultiAgentNR}
N.~Paul, T.~Wirtz, S.~Wrobel, and A.~Kister, ``Multi-agent neural rewriter for
  vehicle routing with limited disclosure of costs,'' \emph{ArXiv}, 2022.

\bibitem{Guo2022}
X.~Guo and B.~Hu, ``Exact formulas for finite-time estimation errors of
  decentralized temporal difference learning with linear function
  approximation,'' \emph{ArXiv}, 2022.

\bibitem{Doan2019}
T.~Doan, S.~Maguluri, and J.~Romberg, ``Finite-time analysis of distributed
  {TD}(0) with linear function approximation on multi-agent reinforcement
  learning,'' in \emph{Proceedings of the 36th International Conference on
  Machine Learning}, vol.~97.\hskip 1em plus 0.5em minus 0.4em\relax PMLR, June
  2019, pp. 1626--1635.

\bibitem{bertsekas2019}
D.~Bertsekas, \emph{Reinforcement Learning and Optimal Control}, ser. Athena
  Scientific optimization and computation series.\hskip 1em plus 0.5em minus
  0.4em\relax Athena Scientific, 2019.

\bibitem{Nedic2009}
A.~Nedic and A.~Ozdaglar, ``Distributed subgradient methods for multi-agent
  optimization,'' \emph{IEEE Transactions on Automatic Control}, vol.~54,
  no.~1, pp. 48--61, 2009.

\bibitem{Gatsis22}
K.~Gatsis, ``Federated reinforcement learning at the edge: Exploring the
  learning-communication tradeoff,'' in \emph{2022 European Control Conference
  (ECC)}, 2022, pp. 1890--1895.

\bibitem{bertsekas1996}
D.~Bertsekas and J.~Tsitsiklis, \emph{Neuro-Dynamic Programming}.\hskip 1em
  plus 0.5em minus 0.4em\relax Athena Scientific, 1996.

\bibitem{VertovecCode2022}
\BIBentryALTinterwordspacing
N.~Vertovec, 2022. [Online]. Available:
  \url{https://github.com/nikovert/distributed\_traffic\_routing}
\BIBentrySTDinterwordspacing

\end{thebibliography}

\end{document}